\documentclass[11pt]{amsart}

\usepackage{epsfig}
\usepackage{amsfonts}
\usepackage{amssymb}
\usepackage{amsmath}
\usepackage{amsthm}
\usepackage{amscd}
\usepackage{amstext}
\usepackage{latexsym}

\usepackage{enumerate}

\usepackage{xy}
\usepackage{footmisc}

\input xy
\xyoption{all}

\newcommand{\CC}{\mathbf{C}}

\newcommand{\PP}{\mathbf{P}}
\newcommand{\JJ}{\mathcal{J}}

\newcommand{\OO}{\mathcal{O}}

\newcommand{\iddb}{\sqrt{-1} \partial \overline{\partial}}

\newcommand{\db}{\overline{\partial}}
\newcommand{\al}{\alpha}

\newcommand{\ep}{\epsilon}

\newcommand{\qa}{\quad}
\newcommand{\vp}{\varphi}

\newcommand{\noi}{\noindent}

\providecommand{\abs}[1]{\left|#1\right|}

\theoremstyle{plain}

\newtheorem{theorem}{Theorem}[section]

\newtheorem{proposition}[theorem]{Proposition}
\newtheorem{definition}[theorem]{Definition}

 \newtheorem{example}[theorem]{\textnormal{\textbf{Example}}}

\newtheorem{Standard Process}[theorem]{Procedures}

\newtheorem{remark}[theorem]{Remark}

\theoremstyle{remark}
\newtheorem{remark1}[theorem]{Remark}

\DeclareMathOperator{\ord}{ord}

\DeclareMathOperator{\Span}{span}

\begin{document}

\title{Skoda division of line bundle sections and pseudo-division}



\author{Dano Kim}

\date{}

\maketitle

\begin{abstract}

\noindent We first present a Skoda type division theorem for holomorphic sections of line bundles on a projective variety which is essentially the most general, compared to previous ones. It is derived from Varolin's theorem as a corollary. Then we revisit Geometric Effective Nullstellensatz and observe that even this general Skoda division is far from sufficient to yield stronger Geometric Effective Nullstellensatz such as `vanishing order $1$ division', which could be used for finite generation of section rings by the basic finite generation lemma. To resolve this problem, we develop a notion of pseudo-division and show that it can replace the usual division in the finite generation lemma. We also give a vanishing order 1 pseudo-division result when the line bundle is ample. 
\end{abstract}

\footnotetext{\noindent 

\noi Department of Mathematical Sciences, Seoul National University

\noi Gwanak-ro 1, Gwanak-gu, Seoul, Korea 151-747

\noi Email address: kimdano@snu.ac.kr

This work was supported by the  National Research Foundation of Korea grants NRF-2012R1A1A1042764 and No.2011-0030795, funded by the Republic of Korea government.}

%

\section{Introduction}

 Let $X$ be a smooth projective variety of dimension $n \ge 1$ and let $L$ and $G$ be line bundles on $X$. Let $g_1, \cdots, g_p$ be holomorphic sections in $H^0(X, G)$. If a holomorphic section $f \in H^0 (X, L + G)$ (where we use additive notation for the tensor power of line bundles) is written as $f = g_1 h_1  + \cdots  + g_p h_p$ for some $h_1 , \cdots, h_p \in H^0 (X, L)$, then we say $f$ is \textbf{divided} by $g_1, \cdots, g_p$. 
 
  Skoda-type global division theorems~\cite{Sk78,D, EL, O04, S06, EP} provide such division under certain efficient positivity conditions when we set $L$ as an adjoint line bundle (see Theorem~\ref{div0}). 
Its criterion for division is given in terms of \textit{multiplier ideal sheaves} of the singular weights at hand, which is natural in view of its (analytic) method of proof, $L^2$ methods for $\db$. (For other versions and methods on Skoda-type division, see Section 2.) 

 On the other hand, Geometric Effective Nullstellensatz(GEN for short)~\cite{EL} is a type of division theorem where its criterion for division is given in terms of \textit{vanishing order} along components of the common zero set of $g_1, \cdots, g_p$. As the name suggests, when $X$ is a complex projective space, GEN recovers  the usual algebraic effective Nullstellensatz (see the excellent introduction of \cite{EL}). GEN can be seen as a consequence of Skoda-type division. (See \cite{AW}, \cite{AW1} and the references therein for a different approach.)

 In this paper, we revisit these two types of division,  with  applications to  finite generation of section rings in mind. For Skoda division, we present the most general statement of the kind, \textbf{Theorem~\ref{maind}}, for line bundle sections on a smooth projective variety using Varolin's division theorem~\cite{Vd}. We carefully compare it with previous results in Section 2.

 More importantly for GEN-type division~\cite{EL}, we first note that the main Theorem (ii),(iii) of \cite[p.430]{EL} states the following in its notation : let $\JJ$ be the ideal sheaf generated by $g_1, \cdots, g_p$. Let $Z_1, \cdots, Z_t$ be the distinguished subvarieties with coefficients $r_i > 0$ determined by $\JJ$. Theorem (ii),(iii) above then says that a section which vanishes to order $ \ge \min(n+1, p) \cdot r_i $ at the general point of each $Z_i$ is divided by $g_1, \cdots, g_p$.  \footnote{Elizabeth Wulcan kindly informed the author that $\min(n+1,p)$ of \cite{EL} is reduced to $\min(n,p)$ in \cite[Corollary 7.2]{AW1}.  See also Example~\ref{ex11}. }

  This leaves the important question : does there exist an improvement on the vanishing order criterion by lowering the factor $\min(n+1, p)$?   Indeed, \cite[p.431]{EL} mentions that the results of
Koll\'ar and Brownawell on effective Nullstellensatz might suggest the hope that $\min(n+1, p) \cdot r_i $ could be reduced to $r_i$, but also explains that such hope has a counterexample as in Example~\ref{einlaz} below. 

 On the other hand, such reduction of $\min (n+1, p)$ to $1$ is indeed possible in some concrete examples (Example~\ref{ex11}). In general, if there is only one distinguished subvariety $Z_1$ and if such reduction of $\min (n+1, p) \cdot r_1$ to $r_1$ is possible, we will call such division as \textbf{vanishing order $1$ division}.

  Vanishing order $1$ division could have an immediate application toward finite generation of section rings as in Proposition~\ref{fglemma} if  $\rho$ is given as restriction of line bundle sections to the only distinguished subvariety $Z_1$.  
    
       The presence of the factor $\min(n+1, p)$ is due to its appearance in Skoda-type division as the lower bound of the exponent $q+1$  in Theorem~\ref{div0}. Since this exponent is depending on $p$, the number of sections, Skoda division is apparently missing some possible instances of division (e.g. see Example~\ref{limited}). Note that a role is played here by the obvious fact that in order to define a subvariety of codimension $k$, one needs more than $k$ equations, unless it is a complete intersection. 
       
        The lower bound $\min(n+1, p)$ is used in the version of $L^2$ methods for $\db$, the Skoda estimate, which is the key ingredient of the analytic proof of Theorem~\ref{div0}.  It does not seem possible to do better than $\min(n+1, p)$ in the Skoda estimate in general.

  Instead, we propose the following weakened version of division which seems to be the right notion to work with for such improvement.

\begin{definition}\label{psdiv}

 We say that $f \in H^0 (X, L + G)$ is \textbf{pseudo-divided} by $g_1, \cdots, g_p \in H^0(X, G)$ if there exist a line bundle $W$ on $X$ and nonzero
 holomorphic sections $w_1, \cdots, w_t \in H^0(X, W)$ such that for each $1 \le i \le t$, there exist sections
  $h_{ijk} \in H^0(X, L)$ ($j = 1, \cdots, p$, $k = 1, \cdots, t$) satisfying

\begin{equation}\label{pdiv}
  f w_i = g_1 H_{i1} + g_2 H_{i2}  + \cdots + g_p H_{ip}  
\end{equation}  
    where $H_{ij} := \sum_{k=1}^t w_k h_{ijk} $.

\end{definition}

  We were led to this definition while following a very natural idea of Forster and Ohsawa~\cite[p.96]{FO}.  Their setting is holomorphic functions on an affine variety where one can always take $w_1 v_1 + \cdots + w_t v_t = 1$ for some $v_i$'s. Therefore the current setting is much more general in that we have to deal with the presence of supplementary sections $w_k$'s rather than simply removing them. 
  
 It is clear that if $f$ is divided by $g_1, \cdots, g_p$, then $f$ is pseudo-divided by them.  The two notions are different as shown by the following example. We also note that in this example, vanishing order $1$ pseudo-division indeed holds.

\begin{example}\cite[Example 2.3]{EL}\label{einlaz}

Let $ X = \CC^2$ with coordinates $x,y$ and let $a=2$. Let $g_1 = x^2, g_2 = y^2$. In this case, the distinguished subvariety is $Z_1 = \{ (0,0) \}$ with coefficient $r_1 = 2$. 

 Let $f = xy$.   If we choose $w_1 = x, w_2 = y$, we get $f w_1 = g_1 w_2$ and $f w_2 = g_2 w_1$. Thus $f$ is pseudo-divided by $g_1, g_2$, but not divided by $g_1, g_2$. One can do simiarly for $a > 2$. 

\end{example} 

 It is easy to see that if $f$  is pseudo-divided by $g_1, \cdots, g_p$, then locally $f$ belongs to the integral closure of the ideal sheaf generated by $g_1, \cdots, g_p$. Note that the integral closure of $(x^a, y^a)$ is $(x,y)^a$ in the above example. 

 As an application of the notion of pseudo-division, we first show that in the finite generation lemma, Proposition~\ref{fglemma}, one can replace division by pseudo-division (\textbf{Theorem~\ref{fglemma2}}). This can be used in yet another approach to finite generation of canonical rings (\cite{BCHM} also cf. \cite{S06}, \cite{CL}, \cite{P}, \cite{BP},  \cite{B} and the references therein).

  As for GEN with pseudo-division, we obtain the following theorem in the simplest case of $g_1, \cdots, g_p$ when they generate the ideal sheaf of a smooth subvariety $Z \subset X$. This is a general result of vanishing order $1$ pseudo-division when the line bundle is ample.

\begin{theorem}\label{pdiv1}
Let $Z \subset X$ be a smooth subvariety of a smooth projective variety and suppose that $g_1, \cdots, g_p \in H^0 (X, G \otimes I_Z)$ are generating sections of the sheaf $G \otimes I_Z$ for a line bundle $G$. Suppose that $G$ is ample. Then there exists $m_0 \ge 1$ such that if $m \ge m_0$ and $f \in H^0 (X, K_X + mG)$ vanishes along $Z$ with order $\ge 1$, then $f$ is pseudo-divided by $g_1, \cdots, g_p$. 

\end{theorem}

 For an alternative approach to finite generation of canonical rings, what one needs in addition to Theorem~\ref{pdiv1} is to allow $G$ to be a big line bundle and keep track of its base locus in the method of proof of Theorem~\ref{pdiv1}. It is well-known that finite generation of canonical rings is reduced to the case of $K_X$ big. Then such vanishing order $1$ pseudo-division for a big line bundle can be combined with the use of Theorem~\ref{fglemma2} and extension results as in \cite{K1}. 
 
 For the general case of $g_1, \cdots, g_p$,  see Remark~\ref{general} about pursuing the analogue of Theorem~\ref{pdiv1}. 

\qa

\noi \textbf{Acknowledgements}

\noi 
The author wishes to express his gratitude to Lawrence Ein for first pointing out that
in the algebraic setting, Skoda division in the case of $\alpha = 1$ can be reduced to the
case of  $\alpha  > 1$. He also wishes to thank Mihai P\u aun and Chen-Yu Chi for comments
during a talk on Theorem~\ref{maind} which helped remove slight restriction from its conditions and Yum-Tong Siu and Elizabeth Wulcan for helpful comments.

\section{Comparison of division theorems in the projective case}

 The original Skoda division theorem~\cite{Sk72} was proved for a pseudoconvex domain in $\CC^n$. There have been extensive developments since \cite{Sk72}. The first compact K\"ahler version  was given in \cite{Sk78}. We note that in the compact case, it is easier to compare results with different curvature conditions and also with concrete examples. Moreover, we restrict to the projective case since regularization of a singular hermitian metric has not been available to the full generality in the compact K\"ahler case. \footnote{A recent work of Junyan Cao~\cite{C} may shed light on that case. Also recently Xu Wang informed the author of his ongoing work on that case related to \cite{Ch}. } Also we restrict to the line bundle formulation which is standard and sufficient for our applications. For the setting of vector bundle morphisms and exact sequences, see e.g. \cite{Sk78, D, J1, J2, K2}. Also note that \cite{EP} is in the generality of division in $H^q (q \ge 0)$.

 In this section, we state Theorem~\ref{maind} and compare it with previous results \cite{Sk78}, \cite{D},
 \cite{EL}, \cite{O04}, \cite{S06}, \cite{EP} for line bundle sections on a smooth projective variety.

\begin{theorem}\label{maind}

 Let $X$ be a smooth projective variety and $G$ a line bundle on $X$. Let $g_1, \cdots, g_p \in H^0(X,G)$
  be holomorphic sections ($p \ge 1$).  Let $(G,e^{-\eta})$ and $(H,e^{-\psi})$ be  singular hermitian metrics with nonnegative
   curvature current for $G$ and for another line bundle $H$ on $X$, respectively.  Let $q \ge \min (n, p-1)$ be an integer where $n = \dim X$. 
     If a holomorphic section $f \in H^0 (X, K_X + H + (q+2) G)$ satisfies

 $$ I := \int_X \abs{f}^2   \frac{1}{\abs{g}^{2(q+1)}  }       e^{-\eta} e^{-\psi}   < \infty  \text{\;\; where \;\;} \abs{g}^2 = \sum_{i=1}^p \abs{g_i}^2  $$

\noindent then there exist $h_1, \cdots, h_p \in H^0(X, K_X + H + (q+1)G)$ such that

$$ f = g_1 h_1 + \cdots + g_p h_p  . $$

\end{theorem}

\noi Note that this statement does not come with the usual $L^2$ estimate against $I$ for $h_i$'s
such as

\begin{equation}\label{estim1}
  \int_X \abs{h_i}^2 \frac{1}{ \abs{g}^{2q} } e^{-\eta}
e^{-\psi} \le C \int_X \abs{f}^2   \frac{1}{\abs{g}^{2(q+1)}  }       e^{-\eta} e^{-\psi}  \end{equation}
 for some constant $C$.  This is because an auxiliary factor $\frac{1}{\tau + A}$ could not be removed from the estimate in the proof.
 
 Theorem~\ref{maind} follows from the important work of \cite{Vd} which reduced the \emph{usual} constant of Skoda 
 $\alpha > 1$ to $\alpha = 1$ using the twist of $\db$. Here $\alpha  \ge 1$ is the exponent of the singular weight $(\frac{1}{\abs{g}^2} )^{(q+1) \alpha}$ as in the above statement.  We mention that the very recent proof of the Strong Openness Conjecture~\cite{GZ} provides an alternative way of reducing $\al >1 $ to $\al =1$.   
 
  The following example shows that one cannot strengthen Theorem~\ref{maind} by replacing the line bundle and the weight
$ ( (q+2)G, \frac{1}{\abs{g}^{2(q+1)} } e^{-\eta} ) $ by $( (q+1)G, \frac{1}{\abs{g}^{2(q+1)} } ) $.

\begin{example} \cite{O04},\cite{S07} Let $X = \PP^1$ and take $H = \OO(0)$, $G = \OO(1)$. Let $g_1, g_2 \in H^0(X, G)$ be the two standard generating sections. Since $q= \min(1, 1) = 1$ in this case, we have $K_X + H + (q+1)G = \OO(0)$. The constant section of $\OO(0)$ satisfies the norm condition with respect to $\log (\abs{g_1}^2 + \abs{g_2}^2)$, but the division statement is not true since it will imply the existence of a global holomorphic section of $K_X + H + qG = \OO(-1)$ on $X$.

\end{example}

  Next, we make a careful comparison of Theorem~\ref{maind} with the
 previous global division theorems of Skoda type and conclude that it is new and essentially the most general one.

\begin{theorem}\label{div0}\cite{Sk78, D, EL, S06, EP}  In each case of the table below, we have division in the following sense: for every section $f \in H^0 (X, K_X  + (q+1)G + q' G + H)$ satisfying

$$ \int_X \abs{f}^2   \left( \frac{1}{\abs{g}^2} \right)^{q+1} e^{-\vp} e^{-\psi}   < \infty       $$

\noindent there exist $h_1, \cdots, h_p \in H^0(X, K_X + H + qG + q' G)$ such that $f = \sum h_i g_i$.

\end{theorem}

\begin{tabular}{| l | l | l | l | l | c | c |}
\hline
    & Results & $q \ge$ & $q'$ &  $e^{-\vp}$ for $q' G$ & $H$ &  $e^{-\psi}$ for $H$ \\ \hline
   (1) &  \emph{\cite{Sk78},\cite{D}, ...} & $p'$ &  1  &     $ (\frac{1}{\abs{g}^2})^{\ep} (e^{-\eta})^{1-\ep}$        & pseudoeffective & arbitrary \\ \hline
   (2) & \emph{\cite[2.4]{S06}}  & n & 1 &  $\frac{1}{\abs{g}^2}$    & pseudoeffective & arbitrary  \\ \hline
   (3) &  \emph{\cite[1.1(ii)]{EL}} & $p'$  & $ 0 $  & none  & nef and big & minimal sing.   \\  \hline
   (4) & \emph{\cite[4.1]{EP}}   & n & 1 &  $\frac{1}{\abs{g}^2}$      & "nef and good"  & see  \cite[4.1]{EP}   \\ \hline
   (5) &  Theorem~\ref{maind} &  $p'$   & 1  & $e^{-\eta}$ & pseudoeffective& arbitrary \\ 
        \hline

\end{tabular}

\qa

\noindent  Here $p'$ denotes $\min (n,p-1)$ and $e^{-\eta}$ is an arbitrary singular hermitian metric with nonnegative curvature current. The term `minimal sing' refers to the singular hermitian metric with minimal singularities in the sense of Demailly.  Strictly speaking, Case (1) statement is not contained in \cite{Sk78},\cite{D}, but can be proved by their methods using (\ref{sdsd}).  We mention that in the Cases (3) and (4), $(H, e^{-\psi})$ is assumed to be algebraic and the proof
  is accordingly algebraic in the sense that it uses cohomology vanishing and the exactness of Skoda complexes (see \cite{L}).  
  
  We will see shortly that Case (5) contains all the other cases except that in Case (3), $q' =0$ is allowed. However, this is at the expense of requiring $H$ to be nef and big, which is much stronger than pseudoeffectivity. In this sense, Theorem~\ref{maind} is essentially the most general. 
  
 Case (5) contains Cases (3) and (4) as seen from the following proposition and the fact that  a nef and good line
  bundle possesses a singular hermitian metric with its Lelong number zero at every point \cite{Ru}. On the other hand, a line bundle possessing a metric with everywhere zero Lelong number may not be nef and good (see the example of \cite[p.145]{DEL} and \cite[Example 5.3]{F} for a semipositive line bundle with $\kappa = 0, \nu =1$). 
  
\begin{proposition}\cite{K3}   Let $\vp$ and $\psi$ be two plurisubharmonic functions on a complex manifold $X$. Suppose that $\psi$ has zero
 Lelong number at every point of $X$. Then we have
 $$ \JJ( \vp + \psi ) = \JJ (\vp ) .$$

\end{proposition}

  Finally, we also point out the following strengthening of \cite[Corollary 4.4]{EP} which had required $B$ to be nef and good.
The condition on $B$ as in Proposition~\ref{44} is much weaker than being nef and good (cf. \cite[Example 2.3.7]{L}).

\begin{proposition}\label{44}

 Let $X$ be a smooth projective variety and $L = K_X + B$ an adjoint line bundle. Suppose that the line bundle $B$
 admits a singular metric $(B, h)$ with nonnegative curvature such that $\JJ(h) = \OO_X$. Suppose also that $L$ is
 generated by global sections. Then the section ring $R(L)$ is generated by $ \oplus_{m \le n+2} H^0 (X, mL)$.

\end{proposition}

\noi This is an easy consequence of \cite[Theorem 4.1]{EP} (or Theorem~\ref{div0}).

\qa

\section{Proof of Theorem~\ref{maind}}

Before the proof, we recall the following fundamental steps used  in the method of $L^2$ estimates for a normal projective variety $X \subset \PP^N$.

 \begin{Standard Process}[From Stein to Projective]\label{sdsd}

\qa

 \begin{enumerate}[(a)]

 \item

 Choose a hyperplane section $H \subset \PP^N$ where $X$ is embedded in $\PP^N$ so that $X \setminus H$ is an
 affine variety. Let $H$ contain the singular locus of $X$.

\item

 Write the Stein manifold $X \setminus H$ as an infinite union of $\cup_{t \ge 1} \Omega_t$ where $\Omega_t$ is
 relatively compact Stein open subset in $\Omega_{t+1}$.

\item

 Apply one instance of H\"ormander's  $L^2$ existence on $\Omega_t$.

\item

 Use Montel type lemma and $L^2$ Riemann extension lemma (see e.g. \cite[Prop. 2.17 and 2.18]{K1}    
  \footnote{which contain the typos of `multi-valued section', which should be
removed.}) taking the limit as $t \to \infty$.

 \end{enumerate}

 \end{Standard Process}

\noindent One reason we need (\ref{sdsd}) is because we need to regularize a singular hermitian metric on the sequence
of
 $\Omega_t$'s. Therefore, we need $X$ to be projective, not just compact K\"ahler. Since we can let $H$ include the singular
 points of $X$, we can formulate $L^2$ extension and division theorems on normal projective varieties. Normality is needed to extend holomorphic sections on the regular part of $X$ to the entire $X$. In this paper, we mostly restrict to smooth $X$ following \cite{EL}. 

\begin{remark}

 To use the above procedures in a possibly non-projective setting, \cite{Vd} introduces the notion of $X$ being
 essentially Stein, which means that for a compact K\"ahler $X$, there exists a closed analytic subvariety $Z$ such
 that $X \setminus Z$ is Stein. We note that there is indeed an example of non-projective, essentially Stein compact
 K\"ahler manifold in \cite{HV} in the context of the problems of compactifying $\CC^n$.

\end{remark}

\noindent \textit{Proof of Theorem~\ref{maind}.} We will use Varolin's main division theorem,   stated here for the  convenience of the reader (see \cite{Vd} for the full details on its setting).

\begin{theorem}\cite[Theorem 1]{Vd}\label{varolin}
 Let $X$ be a smooth projective variety. Let $(L, e^{-\psi})$ and $(G, e^{-\eta})$ be two singular hermitian line bundles on $X$. Let $g_1, \cdots, g_p \in H^0(X, G)$. Let

 $$ q = \min(p-1, n)   \text{  and  }    \xi = 1 - \log( \abs{g}^2 e^{-\eta} )   .$$

\noi Let $(\phi, F, q)$ be a Skoda triple. Set $\tau, A, B$ as in \cite{Vd}. Assume that

\begin{equation}\label{curv}
\iddb \psi \ge qB \iddb \eta
\end{equation}
 and that

\begin{equation}\label{eta1}
 \abs{g}^2 e^{-\eta} <    1  .
\end{equation}

\noi Then for every section $f \in H^0 (X, K_X + L)$ such that

$$ \int_X   \abs{f}^2  \frac{B}{\tau (B-1) }    \frac{1}{ (\abs{g}^2 e^{-\eta} )^{q+1}  }  e^{ \phi (\xi) } e^{-\psi} < \infty  ,$$

\noi  there exist sections $h_1, \cdots, h_p \in H^0(X, K_X + L - G)$  satisfying

 $$ f = \sum h_i g_i $$ and

\begin{equation}\label{estimate}
 \int_X \abs{h}^2 \frac{1}{\tau + A}    \frac{1}{ (\abs{g}^2 e^{-\eta} )^{q}  }  e^{ \phi (\xi) } e^{-\psi - \eta}
  \le     \int_X   \abs{f}^2  \frac{B}{\tau (B-1) }    \frac{1}{ (\abs{g}^2 e^{-\eta} )^{q+1}  }  e^{ \phi (\xi) }
  e^{-\psi}.
\end{equation}

\end{theorem}

\qa

 To prove Theorem~\ref{maind}, we need to make Theorem~\ref{varolin} comparable to the setting of Theorem~\ref{div0}. First of all, the norm criterion above comes with extra factors such as $  \frac{B}{\tau (B-1) } ,  e^{-\eta},  e^{ \phi (\xi) } $.
 Their singularities should be examined before we could use this norm criterion to divide sections. Note that each of $\tau$,
 $A,B$ (used for twisting of $\db$) is a function of $\lambda$ where $\lambda = C - \log (\abs{g}^2 e^{-\eta}) + \ep^2 $ for appropriate constants $C$
 (depending on $\Omega_t$).
 \\

 First, we consider the conditions for the metric $e^{-\eta}$. In applications, the singular weight given by a power
of $\frac{1}{\abs{g}^2}$ carries a clear
 geometric meaning, and we do not want to alter it by singularity of the auxiliary metric $e^{-\eta}$. So
 it would be nice to take $e^{-\eta}$ as a smooth metric if possible.

 On the other hand, we have the curvature condition \eqref{curv} where $B$ is a function for which it is
  known that $qB \le q + 2$. Since $qB$ is not a constant, it is in general
hard to satisfy $\eqref{curv}$ unless $\iddb \eta \ge 0$. If $\iddb \eta \ge 0$, we can let $L$ contain a multiple of
$G$ so that $\eqref{curv}$ is satisfied.

 If we require $\eta$ to satisfy $\iddb \eta \ge 0$ and also to be smooth, it is too restrictive. The best compromise
 is to let $\eta$ be a metric \emph{less singular} (defined as  in \cite{D})  than $\log \abs{g}^2$ with nonnegative curvature current $\iddb \eta$.
 Then \eqref{eta1} is also automatically satisfied. Theorem~\ref{maind} for general $\eta$ without necessarily being less singular than   $\log \abs{g}^2$  then follows:  take $\eta' = \sup (\eta, \log \abs{g}^2)$ which defines a singular hermitian metric of $G$ less singular than $\log \abs{g}^2$. The given section $f$ in Theorem~\ref{maind} satisfies $I < \infty$ with respect to $e^{-\eta}$ and therefore also with respect to $e^{-\eta'}$.

\qa

Second, we consider  the conditions on $\tau, B, A$.  It is enough to take $ \phi = 0$. It
  would be best to have constants $C_1$ and $C_2$ such that $\displaystyle 0 < C_1 \le \frac{B}{\tau (B-1)}  \le  C_2  \quad \text{ and } \qa 0 < C_1 \le \frac{1}{\tau + A} \le C_2  .$  It is possible to obtain the first condition for $\frac{B}{\tau (B-1)}$, but the author does not know
whether a choice of $\tau, A, B$ for the second one is possible. Instead, for our purpose in this paper, it is enough
to take, following \cite{MV}, $\tau (x) = 2 + \log (\frac{2}{e} e^x - 1) $, $A(x) = \frac{2}{e} e^x$ and $B(x) = 1 +
\frac{1 + F'(x)}{q \tau (x)} $. Then $\frac{B}{\tau (B-1)} = \frac{1}{\tau} + \frac{q}{ 1 + F'} =  \frac{1}{\tau} +
\frac{q}{\tau'}$. Since $\tau \ge 2$ and $2 \ge \tau' \ge 1$, we get $\displaystyle \frac{1}{2} + q \ge \frac{1}{\tau} + \frac{q}{\tau'} \ge \frac{q}{2}  $.
 In other words, $\frac{B}{\tau (B-1)}$ cannot have \textit{a pole or a zero}. Similarly, $\frac{1}{\tau + A}$ cannot have a pole (i.e. going to $+\infty$) since $\tau \ge 2$ and $A \ge
\frac{2}{e}$.

 Now let us examine what happens when $\tau + A$ becomes large. Note that there are two constants $C', C''$ such that $x +
 C' \le \tau (x) \le x + C''$. When  $\frac{1}{\tau + A}$ is taken together with $\frac{1}{ (\abs{g}^2)^q }$ (when $q \ge 1$) , it is bounded away
from zero \cite[p.1443]{K1}. Our $x$ is of the form $x = C - \log \abs{g}^2$ for an appropriate constant $C > 0$. It is sufficient to
consider the limit of the function
$$ \frac{1}{x + e^x} = \frac{ \abs{g}^2 }{ e^C + \abs{g}^2 (C - \log \abs{g}^2 )  }     \to \frac{\abs{g}^2}{e^C}     $$  as $x \to +\infty$
or equivalently as $\abs{g}^2 =: y \to 0$. Therefore
the factor $\frac{1}{\tau + A}$ obtains a zero of the same order as $\abs{g}^2$. This means that we cannot remove the
factor $\frac{1}{\tau + A}$ from \eqref{estimate}. However, since $\frac{1}{\tau + A}$ is combined with the weight
$\frac{1}{ \abs{g}^2 }$ (where $q \ge 1$), we can extend the $h_i$'s we want from $X
\setminus H$ to $X$ as in the standard Procedures~\ref{sdsd}. This concludes the proof of Theorem~\ref{maind}.

\qa

\begin{remark1}

   One might want to see what happens when we formulate Skoda's basic estimate
   replacing $ q \log (\abs{g}^2 e^{-\eta}) $ by $ q \log \abs{g}^2$ and adjusting other
   parts accordingly. This seems a natural thing to try since $q \log \abs{g}^2 $ gives a natural singular
   hermitian metric for the line bundle $q G$, without having to measure it against another metric of $g$,
   that is $e^{-\eta}$. However, it seems that such formulation of Skoda's basic estimate cannot be used 
    since in the inequality \cite{Vd} p.19 (15), we will not be able to combine the two terms $-q \tau (B - 1)$ and $ \tau \phi' (\lambda) + 1 + F' (\lambda)$.

\end{remark1}

\section{Pseudo-division and Finite generation}\label{section 6}

We first observe the following two examples which show that Skoda division theorem even in its greatest generality (as in Theorem~\ref{maind}) does not capture all possible instances of division of line bundle sections. In the first example, note that the factor $\min (p, n+1) =3$ in the vanishing order type division can be indeed lowered to $1$.

\begin{example}\label{ex11}

 Let $X = \PP^2$ and let $G = \OO(2), H = \OO$ with three sections $g_1 = xy - z^2, g_2 = xy - 2z^2 , g_3 = y^2 - xz $
  of $G$. Thus  $p = 3$ and  $q = 2$. All global sections of $K_X +  (q+1) G = \OO(3)$ except $x^3$ are divided by
  $g_1, g_2, g_3$ such as $xy^2 = g_1 (2y) + g_2 (y)  $ and $xz^2 = g_1 (x) + g_2 (-x)   $. Let $a = [1,0,0]$. Since locally $ g_1 = y - z^2  ,  g_2 = y - 2z^2  , g_3 =  y^2 - z   $ around $a$, we see that
 $\JJ ( 2 \log \abs{g}^2)  =  \mathfrak{m}_a  $ and  $\JJ ( 3 \log \abs{g}^2)  =  \mathfrak{m}_a^2  $. Hence we
 conclude
 that a section of $\OO(3)$ only needs to vanish along  $\JJ ( 2 \log \abs{g}^2)$, not along  $\JJ ( 3 \log
 \abs{g}^2)$ in order to be divided by $g_1, g_2, g_3$. Note that the only distinguished subvariety here is the point $Z_1 = a$ and the associated coefficient is $r_1 = 1$.

\end{example}

\begin{example}\label{limited}

 Under the setting of Theorem~\ref{maind}, for some integer $k$ with $1 \le k < p$ and $k \le n$, suppose
 that sections  $g_{k+1}, \cdots, g_p$ are linear combinations of $g_1, \cdots, g_k$. Then the metric given by
 $\log (\abs{g_1}^2 + \cdots + \abs{g_k}^2)$ is equivalent to that given by $\log (\abs{g_1}^2 + \cdots + \abs{g_p}^2)$.
 Therefore in this case, we can divide a section $f$ if it is only $L^2$ with respect to the $k$-th power of
  $\frac{1}{ \abs{g_1}^2 + \cdots + \abs{g_p}^2  }$ which is a weaker condition than the $p$-th power required by Theorem~\ref{maind}.

\end{example}

 Toward finite generation of section rings, the following basic lemma can be used once one has vanishing order $1$ division with one distinguished subvariety $Z_1$ and $r_1 = 1$ as discussed in the Introduction. The map $\rho$ shall be restriction of sections to $Z_1$.

\begin{proposition}\label{fglemma}
 Let $R := \oplus_{m \ge 0} R_m$ and $S := \oplus_{m \ge 0} S_m $ be two graded subrings of complete section rings on proper normal varieties over a field $k$. Let $\rho: R  \to S$ be a grade-preserving surjective morphism over $k$. Suppose that

\begin{enumerate}
    \item $\oplus_{m \ge 0} S_m$ is finitely generated as a graded algebra over $k$, and

    \item There exist sections $d_1, \cdots, d_k \in R_{m_1}$ for some $m_1 \ge 1$ such that

    $$\forall m \ge m_1 \text{\; and \;} \forall s \in R_m \cap \ker \rho \Rightarrow \exists q_1, \cdots, q_k \in R_{m- m_1} \text{ such that } s = \sum d_i q_i .$$

\end{enumerate}

\noi Then $\oplus_{m \ge 0} R_m$ is finitely generated.

\end{proposition}

\begin{proof}

 Let $g_1, \cdots, g_l$ be the generators of $S$ as a graded algebra over $k$. Since $\rho$ is
 surjective, we can choose $g_1', \cdots, g_l'$ such that $g_i' \in \rho^{-1}(g_i)$ for each $i$.
 Now let $y \in R_m$ be an arbitrary homogeneous element of $R$.  The image $\rho(y)$ is written as
  $p(g_1, \cdots, g_l)$ where $p(x_1, \cdots, x_l)$ is a polynomial. Then $y - p(g_1',\cdots,g_l') \in R $ is
  in the kernel of $\rho$, so we can apply the second condition in the statement. By induction on $m$, this shows
  that $R$ is generated as a graded algebra over $k$ by $d_1, \cdots, d_k, g_1', \cdots, g_l'$ and the basis elements
  of $R_1, \cdots, R_{m_1}$ (which are finite dimensional vector spaces over $k$).

\end{proof}

\qa

We will replace division by pseudo-division in Proposition~\ref{fglemma}. To motivate the definition of pseudo-division, let us consider the following idea of  \cite{FO}. Let $f \in H^0(X, K_X + H
+ qG)$ and $g_1, \cdots, g_p \in H^0(X,G)$ be sections of line bundles and suppose that we want to divide $f$ by
$g_1, \cdots, g_p$.
 For example, assume that $Z \subset X$ is a smooth subvariety of codimension $2$ and that $\dim X = 3$. Suppose that
  $Z$ is
defined by $g_1, g_2, g_3$. Since $\langle g_1, g_2 \rangle \supset I_Z$, if a section $f$ just vanishes along $Z$,
we may not have

$$ \int_X \abs{f}^2 \frac{1}{ (\abs{g_1}^2 + \abs{g_2}^2)^2 }  < \infty   .$$

\noindent However, it may be possible that with some auxiliary section $w_3$, we may have

$$ \int_X \abs{f w_3}^2 \frac{1}{ (\abs{g_1}^2 + \abs{g_2}^2)^2 }  < \infty   $$

\noindent so that we can apply the Skoda division and obtain $f w_3 = g_1 y_1 + g_2 y_2$. It would be even better 
if we can further divide $y_1,y_2$ and get $y_1 = w_1 h_{11} + w_2 h_{12} + w_3 h_{13} $ and $y_2 = w_1 h_{21} + w_2 h_{22} +
w_3 h_{23}$ for some $h_{ij}$'s.   In general, we arrive at the following

\begin{definition}\label{psdiv}

 We say that $f$ is \textbf{pseudo-divided} by $g_1, \cdots, g_p$ if there exist a line bundle $W$ and nonzero
 holomorphic sections $w_1, \cdots, w_t \in H^0(X, W)$ such that for each $1 \le i \le t$, there exist sections
  $h_{ijk} \in H^0(X, H)$ ($j = 1, \cdots, p$, $k = 1, \cdots, t$) satisfying

\begin{equation}\label{pdiv}
  f w_i = g_1 H_{i1} + g_2 H_{i2}  + \cdots + g_p H_{ip}  
\end{equation}  
    where $H_{ij} := \sum_{k=1}^t w_k h_{ijk} $.

\end{definition}

\noi Using the usual determinant argument (as in e.g. \cite[Ch.1 Prop. (2.3)]{Ne}), it follows that
  there is a monic polynomial of $f$,  $p(f) \in H^0(X, k(G+ H))$ (for $k\ge 0$) whose coefficients are generated by $g_1, \cdots, g_p$ and $h_{ijk}$'s
  over $k$ such that $p(f) w_i = 0 $ for each $1 \le i \le t$. Since not every $w_i$ is constantly zero, we have $p(f) = 0$
  as an equality in $H^0(X, k(G+ H))$.

 If we choose the linear system $\Span \{w_1, \cdots, w_t\}$ to be base point free, then any section will have finite $L^2$ norm with respect to the
weight given by the set $w_1, \cdots, w_t$  since it has no poles. However in practice, pseudo-division \eqref{pdiv} will follow from Skoda type division such as Theorem~\ref{maind}. For this reason,  we cannot arbitrarily choose $W$ as a nice (i.e. base-point-free or positive) line
bundle, since when you further divide $H_{ij}$ by $w_k$'s, you need some multiple of $G$ to include that
nice line bundle $W$.

 One important motivation for Definition~\ref{psdiv} is that we can replace `division' by `pseudo-division' in the `finite generation
 lemma' Proposition~\ref{fglemma}.
  We will show this in Theorem~\ref{fglemma2} after we recall some basic notions. Working with section rings, the familiar
  algebraic notion of integral extension between rings appears naturally and turns out to be very useful. We have the
  following version of it for a section ring:  let $X$ be a normal projective variety and $F$ a holomorphic line bundle on $X$.

\begin{definition}

 Let $S := \oplus_{m \ge 1} S_m \subset R(F)$ be a section ring. An element $s \in H^0(X, rF)$ is said to be \textbf{integral} over $S$  if it satisfies a monic equation

 $$ s^k + a_{1} s^{k-1} + \cdots + a_{k-1} s + a_k = 0 $$

\noi in $H^0(X, rkF)$ with $a_i \in S_{ir}$, we have $s \in S_r$.

\end{definition}

\noindent Given two graded section rings $S := \oplus_{m \ge 1} S_m \subset T := \oplus_{m \ge 1} T_m \subset R(F)$,
we say $T$ is \textbf{integral over} $S$ if each homogeneous element of $T$ is integral over $S$.  The integral
closure in $R(F)$ exists for $S$.

\begin{theorem}\label{fglemma2}

 In the statement of Proposition~\ref{fglemma}, it is sufficient only to pseudo-divide $s$ by $d_1, \cdots, d_k$.

\end{theorem}

\begin{proof}

 Let $\rho: R := \oplus_{m \ge 0} R_m \to S := \oplus_{m \ge 0} S_m $ be a surjective morphism between two graded algebras over $k$.  Suppose that $\oplus_{m \ge 0} S_m$ is finitely generated as a graded algebra over $k$.

 Let $T$ be the graded subring of $R$ generated by $d_1, \cdots, d_k$, $g_1', \cdots, g_l'$ and
 the basis elements of $R_1, \cdots, R_{m_1}$. We will show that $R$ is integral extension over $T$.
  Consider the fields of fractions $F_T$ and $F_R$. They have the same transcendence degrees because of the integral extension.
    Since the fields of fractions $F_T \subset F_R$ is a finite algebraic extension, it follows that $R$ is a finite $T$ module and hence $R$ is finitely generated over $k$.

 Let $y \in R_m$ be an arbitrary homogeneous element of $R$. We can write $\rho(y)$ as a polynomial $p(g_1, \cdots, g_l)$ of the generators $g_1, \cdots, g_l$ of $S$. For the same polynomial $p$, take $y' := p(g_1', \cdots, g_l')$. Then $y - y' \in \ker \rho$ since $\rho (y - y') = \rho(y) - \rho(y') = 0$. As the result of pseudo-division of $y-y'$, we have a monic polynomial

 $$ (y - y')^d + c_1 (y- y')^{d-1} + \cdots + c_{d-1} (y - y') + c_d = 0  $$

\noi where the coefficient $c_i$ is generated by $d_1, \cdots, d_k$ and the basis elements of $R_j$ for $0 \le j < m $.
 By the inductive hypothesis, $T' := T[ R_1, \cdots, R_{m-1}]$ is a finite $T$ module while $T' [y]$ is a finite $T'$ module. Then $T' [y]$ is a finite $T$ module.

 Let $b_1, \cdots, b_q$ be generators for $T' [y]$ over $T$. Then there exists an integer $k \ge 1$ such that  $y^k = t_1 b_1 + \cdots + t_q b_q$ for some $t_1, \cdots, t_q \in T$ and the $y$ terms in $b_1, \cdots, b_q$ has the exponent less than $k$. It follows that $y$ is integral over $T$, since by induction, every element of $R_1, \cdots, R_{m-1}$ is integral over $T$.

\end{proof}

 The above theorem was about how one can use pseudo-division if it is available. Theorem~\ref{pdiv1} in the Introduction is about how one can obtain pseudo-divison using Skoda division, resulting in vanishing order $1$ pseudo-division for an ample line bundle.

\begin{proof}[Proof of Theorem~\ref{pdiv1}]

 Since $g_1, \cdots, g_p$ are defining equations for $Z$ of codimension, say,  $k \ge 1$, each point $x \in Z$ has a Zariski open neighborhood $U_x$ where there exists a multi-index $I = (i_1 < \cdots < i_k)$ such that  $g_{i_1}, \cdots, g_{i_k}$ define $Z$.   In $U_x$, shrinking to an open subset if necessary, one can find local analytic  coordinates $z_1 = g_{i_1}, \cdots, z_k = g_{i_k}, z_{k+1}, \cdots, z_n$. Therefore the singular weight $$ \Psi_I :=  \left( \frac{1}{\abs{ g_{i_1}}^2 + \cdots + \abs{g_{i_k}}^2 } \right)^k $$ in $U_x$ has poles (i.e. $\Psi_I = \infty$) exactly along $Z \cap U_x$. Outside $U_x$, there may be additional poles along the subscheme $S_I$ corresponding to the multiplier ideal sheaf $\JJ (\Psi_I)$. Because of the additional poles, in general $ \int_X \abs{f}^2  \Psi_I = \infty $.   (As usual for an adjoint line bundle, we understand the integrand as well-defined without choosing a measure, once a hermitian metric of $mG$ is specified. Here the metric is given by the product of $\Psi$ for $kG$ and a choice of a smooth hermitian metric for $(m-k)G$. )
 
  To obtain the pseudo-division as in the statement, we let $W = r G$ where $r \ge 1$ is an integer to be determined later. We need sections $w_i \in H^0(X, W)$ such that   
  
\begin{equation}\label{wi}
   \int_X \abs{f}^2 \abs{w_i}^2 \Psi_I < \infty 
 \end{equation}
   which enables us to apply Theorem~\ref{maind} to the section $f w_i$ as in Definition~\ref{psdiv}.  But we need a base point free linear system of such $w_i$'s  to obtain  $H_{ij} := \sum_{k=1}^t w_k h_{ijk} $ in  Definition~\ref{psdiv}, applying Theorem~\ref{maind} again (note that in the case of base point free linear system, the criterion for division $I  < \infty$ is automatically satisfied). 
   
    For this purpose, we need to choose $w_i$ not to vanish along $Z$, but only vanishing along the rest of the components (i.e. \textit{additional poles}) in the subscheme $S_I = Z (\JJ (\Psi_I))$ above. More precisely, those components are described in terms of primary decomposition of the ideal sheaf generated by $g_{i_1}, \cdots, g_{i_k}$ as in Example 9.6.14 \cite{L}. 
  
  For the moment, for the sake of simplicity, let us only consider the underlying closed subsets of these subschemes. Then we can write $S_I = Z \cup Y_I$ where $Y_I$ is the union of the rest of the components. 
 Since the condition \eqref{wi} defines coherent ideal sheaves on the components of $Y_I$, by choosing sufficiently large integer $r$, there exists a linear system $V_I$ of $w_i$'s satisfying \eqref{wi} whose base locus is $Y_I$ (as a set).  
 
 We only need a finite number of $I$'s to account for all the points $x$ on $Z$ and may assume that we chose a minimal set $T$ of such $I$'s.  Then we can obtain linear systems $V_I \subset H^0 (X, r G)$ for each $I \in T$ keeping the same $r \ge 1$.  Taking the span of $V_I$'s in $H^0 (X, rG)$, we obtain the base point free linear system of $w_i$'s we wanted, i.e. $\cap_{I \in T} Y_I = \emptyset $ for the following reason: suppose that the intersection is not empty and contain a point $y \in X$. Since $g_i (y) = 0$ for every $i$, we have $y \in Z$. But there exists $J=(j_1, \cdots, j_k) \in T$ such that $g_{j_1}, \cdots, g_{j_k}$ define $Z$ in a neighborhood of $y$ as in the beginning.   This completes the proof. 

\end{proof}

Finally we give a remark on the general case of $g_1, \cdots, g_p$ where one has many distinguished subvarieties, say $Z_j$ with $r_j > 0$ for a finite set $T$ of $j$'s.

\begin{remark}\label{general}

Suppose that for each $j \in T$, the vanishing order condition $\ord_{Z_j} f \ge r_j$ is equivalent to the condition 

\begin{equation}\label{haha}
 \abs{f}^2    \left(   \frac{1}{\sum^{k}_{i=1} \abs{g_{j_i}}^2 }   \right)^{k'}  \in L^1_{loc}  
\end{equation} 
  on a Zariski open subset of $X$ where $k' = \min (k, n+1)$. Then one can find auxiliary sections $w_l$'s of $W$ so that \eqref{haha} holds everywhere on $X$ when $f$ is replaced by $f w_l$.  From this point on, it will be more complicated in general than the case of Theorem~\ref{pdiv1}  to make a good choice of such  $w_l$'s and obtain pseudo-division.

\end{remark}

\footnotesize

\bibliographystyle{amsplain}

\qa

\qa

\normalsize

\noi \textsc{Dano Kim}

\noi Department of Mathematical Sciences, Seoul National University

\noi Gwanak-ro 1, Gwanak-gu, Seoul, Korea 151-747

\noi Email address: kimdano@snu.ac.kr

\noi


\end{document}